\documentclass{amsart}

\usepackage{amssymb,amsmath}
\usepackage{tkz-graph}

\usepackage{hyperref}

\newcommand\C{\mathbb C}
\newcommand\R{\mathbb R}
\newcommand\Z{\mathbb Z}
\newcommand\N{\mathbb N}
\newcommand\B{\mathcal B}

\newcommand\Lp{{\mathbf L}^p(T,\lambda)}
\newcommand\Lq{{\mathbf L}^q(T,\lambda)}
\newcommand\Lqmu{{\mathbf L}^q(T,\mu)}
\newcommand\Lone{{\mathbf L}^1(T,\lambda)}
\renewcommand\root{\mathfrak{root}}
\newcommand\parent{\operatorname{par}}
\newcommand\child{\operatorname{Chi}}
\newcommand\dist{\operatorname{dist}}

\newcommand\ds{\displaystyle}

\theoremstyle{plain}
\newtheorem{theorem}{Theorem}[section]
\newtheorem{proposition}[theorem]{Proposition}
\newtheorem{lemma}[theorem]{Lemma}
\newtheorem{corollary}[theorem]{Corollary}
\newtheorem{example}[theorem]{Example}

\theoremstyle{definition}
\newtheorem{defi}[theorem]{Definition}

\title[Hypercyclic Shifts on Trees]{Hypercyclicity of Shifts on
  Weighted \\ ${\mathbf L}^p$  Spaces of Directed Trees}
\author{Rub\'en A. Mart\'inez-Avenda\~no}
\thanks{The author would like to thank the Department of Mathematical
  Sciences at George Mason University for their hospitality while he
  undertook this research. In particular, he would like to thank
  Prof. Flavia Colonna for many illuminating conversations about
  operators defined on infinite trees and for his comments on this
  paper. This research was made possible by the ``Programa de
  Estancias Sab\'aticas en el Extranjero 2015'' of the Consejo
  Nacional de Ciencia y Tecnolog\'ia, M\'exico.
\\
{\tt https://doi.org/10.1016/j.jmaa.2016.08.066}
\\
\copyright 2016. This manuscript version is made available under the CC-BY-NC-ND 4.0
license {\tt http://creativecommons.org/licenses/by-nc-nd/4.0/} 
}
\address{Centro de Investigaci\'on en Matem\'aticas, Universidad
  Aut\'onoma del Estado de Hidalgo, Pachuca, Hidalgo, Mexico.}
\email{rubeno71@gmail.com}

\begin{document}

\begin{abstract}
In this paper, we study the hypercyclicity of forward and backward
shifts on weighted $L^p$ spaces of a directed tree. In the forward
case, only the trivial trees may support hypercyclic shifts,
in which case the classical results of Salas~\cite{salas} apply. For
the backward case, nontrivial trees may support
hypercyclic shifts. We obtain necessary conditions and sufficient
conditions for hypercyclicity of the backward shift and, in the case
of a rooted tree on an unweighted space, we show that these conditions
coincide.
\end{abstract}

\maketitle

\centerline{\it In memory of Jaime Cruz Sampedro, mathematician,
  teacher, colleague, and friend.}

\vskip1cm

\section{Introduction}

A bounded operator on a Banach space is called hypercyclic if there
exists a vector such that its orbit under the operator
is dense in the space. The study of hypercyclicity (in other types of spaces)
can be traced back to the first half of the 20th century, to the
papers of Birkhoff~\cite{birkhoff} and MacLane~\cite{maclane}. The
first example of a hypercyclic operator on a Banach space was given by
Rolewicz~\cite{rolewicz} in 1969, but it was not until the last two
decades of the 20th century that the study of hypercyclicity really
took off. Instead of giving here the detailed history of the advances in
hypercyclicity in the past 35 years, we refer the reader to the
excellent books by Grosse-Erdmann and Peris~\cite{GEP} and Bayart and
Matheron~\cite{BaMa}, where the reader can find more information about
this concept and its importance.

One large source of examples and counterexamples in the study of
bounded operators is the class of weighted shifts. The study of
weighted shifts was initiated in the now classical paper of
Shields~\cite{shields} and continued by many authors. The
characterization of the hypercyclicity of weighted shifts is due to
Salas~\cite{salas} (see the books \cite{BaMa} and \cite{GEP} for an
alternative statement of this characterization). Many other classes of
operators on Banach spaces have been shown to be hypercyclic, under
certain conditions. Two other famous families of operators on Hilbert spaces
that contain hypercyclic operators, are adjoints of multiplication
operators (see \cite{GoSh}) and composition
operators on spaces of holomorphic functions (see, e.g.,
\cite{shapiro}).

The interest of the study of operators on infinite trees is motivated
mainly by the research in harmonic analysis dealing with the Laplace
operator on discrete structures, perhaps initiated in the
papers~\cite{cartier1,cartier2}. In particular, infinite trees can be
seen as the natural discretizations of the hyperbolic disk. Much more
information about these topics can be found in the
papers~\cite{ACE1,ACE2, ACE3, ACE4, CoCo, CoEa1, CoEa2, pavone}. We
should mention that the paper \cite{pavone} studies hypercyclicity for
composition operators defined on the boundary of nondirected trees.

In \cite{JJS}, Jab\l o\'nski, Jung and Stochel initiated the study of
weighted shifts on Hilbert spaces of functions defined on infinite
directed trees. In their paper, they study many operator-theoretic
properties of these operators, such as boundedness, hyponormality,
subnormality and spectral properties. 

Motivated by the work in \cite{JJS}, in this paper we study the
hypercyclicity of shifts on directed trees on the weighted ${\mathbf
  L}^p$ space of a directed tree. In concrete, we show in
Section~\ref{sec_shift} that ``forward'' shifts are never hypercyclic
unless they reduce to the classical cases, in which the
characterization by Salas mentioned above can be applied.  In
Section~\ref{sec_adjoint}, we find a concrete form for the adjoint of
the shift and define what we mean by a ``backward shift''. More
interestingly, we show in Sections \ref{sec_root} and \ref{sec_unroot}
that this backward shift on weighted directed trees, may be
hypercyclic if some conditions are satisfied. In concrete, the main
results of this paper provide necessary conditions and sufficient
conditions for hypercyclicity of the backward shift in the case where
the tree has a root. These two conditions coincide when the space is
unweighted, in which case the hypercyclicity of the operator depends
on a simple property of the tree, that of having no ``free ends''. In
the case of the unrooted tree, we only give necessary conditions and
show an example when these conditions are satisfied. When applied to
the classical backward shifts, all of these conditions reduce to the
ones obtained by Salas.

Before we begin, we should mention that in \cite{JJS}, the authors
study the weighted shift on an unweighted ${\mathbf L}^2$ space of the
tree. In this paper we prefer to concentrate on the unweighted shift
on the weighted spaces ${\mathbf L}^p$ of the tree. The reason for
this is that the results are cleaner in the case of the weighted
space, as is also the case of the classical shifts. Results for the
hypercyclicity of the weighted shift can be obtained by similarity
with the shift of the weighted space, as it is done in, for example,
\cite{BaMa,GEP}. We leave these results as an exercise for the
interested reader.

\section{Definitions and Notation}

In this section, we set the basic definitions and notation needed for
the rest of the paper. We denote by $\N$, $\N_0$, $\Z$, $\R$, $\R_+$ and $\C$
the sets of natural numbers, the nonnegative integers, the integers,
the real numbers, the positive real numbers, and the complex numbers,
respectively. 

\subsection*{Hypercyclicity}
We first state the main definition in this paper and a few comments
about it. After that, we present the main tool used to prove that
operators are hypercyclic.

\begin{defi}
Let $\B$ be a Banach space and $S: \B \to \B$ a bounded operator. We
say that $S$ is hypercyclic if there exists a vector $x\in \B$ such
that
$$
\{ S^n x \, : \, n \in \N_0 \}
$$
is dense in $\B$. The (necessarily) nonzero vector $x$ is called a
hypercyclic vector.
\end{defi}
Observe that if the operator $S: \B \to \B$ is hypercyclic, the Banach
space $\B$ is necessarily separable. Observe also that, if $x$ is a
hypercyclic vector for $S$, then, for each $n\in\N$, the vector $S^n
x$ is also a hypercyclic vector and hence the hypercyclic vectors form
a dense set (in fact, as is well-known, they form a dense
$G_\delta$-subset of $\B$).

One does not need to explicitly find a vector $x$ that satisfies the
definition above to show that an operator $S:\B\to \B$ is
hypercyclic. The following theorem gives an extremely useful
sufficient condition for hypercyclicity.

\begin{theorem}[Hypercyclicity Criterion]\label{the_hc}
Let $\B$ be a separable Banach space and $S:\B \to \B$ a bounded
operator. Assume there exists a dense subset $X \subseteq \B$, an
increasing sequence of natural numbers $\{ n_k \}$, and a sequence of
functions $T_{n_k}: X \to \B$ such that
\begin{enumerate}
\item $S^{n_k} x \to 0$ for each $x \in X$,
\item $T_{n_k} x \to 0$ for each $x \in X$, and
\item $S^{n_k} T_{n_k} x \to x$ for each $x \in X$.
\end{enumerate}
Then $S$ is a hypercyclic operator.
\end{theorem}

There are several versions of this criterion, which is also referred to
as the Kitai--Gethner--Shapiro Criterion (for the history of this
criterion and its importance in the development of the field, the
reader is referred to the {\em Sources and Comments} section of Chapter 3
in \cite{GEP}).  A proof of the version presented here can be found in
\cite[Theorem 1.6]{BaMa} and \cite[Theorem 3.12]{GEP}.

\subsection*{Trees}
We now state the relevant definitions and notations for directed trees
and directed graphs introduced in \cite{JJS}, which we will use in
this paper (we assume the reader is familiar with the basic definition
of a graph). A {\em directed graph} $G=(V,E)$ is a pair consisting of
a countably-infinite set $V$ (called the set of {\em vertices} of $G$)
and a subset $E$ of $V \times V$ (called the set of {\em directed
  edges}). The {\em indegree} of a vertex $v$ is the cardinality of
the set $\{ u \in V \; : \,  (u,v) \in E \}$ and the {\em outdegree}
of $v$ is the cardinality of the set  $\{ u \in V \; : \,  (v,u) \in E
\}$.  We only consider graphs $G=(V,E)$ which are {\em locally finite};
i.e., both the indegree and the outdegree of each vertex are finite.

A directed graph has a {\em directed circuit} if there exist $n \in
\N$, $n \geq 2$, and a set of distinct vertices
$$
\{ u_1, u_2, \dots, u_n \},
$$
such that $( u_j, u_{j+1} ) \in E$ for each $j=1, 2, \dots, n-1$, and
$(u_n,u_1)\in E$. Furthermore, we also say that the directed graph has a directed circuit if there exists $v \in V$ with $(v,v) \in E$.

The {\em underlying graph} of a directed graph $G=(V,E)$ is the graph
$\tilde{G}=(V,\tilde{E})$ with vertex set $V$ and (undirected) edge
set $\tilde{E}:=\{ \{u,v\} \, : \, (u,v) \in E\}$. Recall that, given
distinct vertices $v$ and $w$ in a graph $\tilde{G}$, a {\em path} is
a finite set of distinct vertices 
$$
\{ v=u_1, \, u_2, \,  u_3, \, \dots, \, u_n=w \}
$$ 
such that $\{ u_j, u_{j+1} \} \in \tilde{E}$ for each $j=1, 2, \dots, n-1$; we say such a path has {\em length} $n$. A graph $\tilde{G}$ is
{\em connected} if for each pair of vertices $v$ and $w$ there exists
a path between $v$ and $w$. We denote by $\dist(v,w)$  the length of the shortest path in the
undirected graph joining $v$ and $w$ and by $\dist(v,H)$ the infimum
of the distances of $v$ to the nonempty set of vertices $H$.

We can now say what we mean by a directed tree.

\begin{defi}
A directed graph $T=(V,E)$ is a {\em directed tree} if 
\begin{itemize}
\item it has no directed circuits,
\item the underlying graph of $T$ is connected, and
\item the indegree of every vertex is either zero or one.
\end{itemize}
If $T=(V,E)$ is a directed tree, we say that $v \in V$ is a {\em root}
if its indegree is zero. A vertex $v \in V$ is called a {\em leaf} if
its outdegree is zero.
\end{defi}

It can be easily seen \cite{JJS} that if a directed tree has a root, then the
root is unique. If there is a root, we call the tree {\em rooted} and we
denote the root by the symbol $\root$. If there is no root, we say the
tree is {\em unrooted}.

We also need to set some notation. The third condition in the
definition above is necessary for the following definition to make
sense.

\begin{defi}
Let $T=(V,E)$ be a directed tree. Given a vertex $v \in V$, $v\neq
\root$, we define its {\em parent} as the unique vertex $u$ such that
$(u,v) \in E$ and we denote it by  $u:=\parent(v)$. For an integer $n \geq 2$, we
inductively define the operator $\parent^n$ as
$\parent^n(v):=\parent(\parent^{n-1}(v))$, whenever
$\parent^{n-1}(v)\neq \root$. In such a case, we say that $v$ has an
$n$-ancestor and we denote the set of all vertices that have
$n$-ancestors as $V^n$. Also, if $u=\parent(v)$ we say that $v$ is a
{\em child} of $u$ and we denote the set of children of $u$ by
$\child(u)$. For $n\in \N$ and $u \in V$, we also define the set
$$
\child^n(u) := \{  v \in V \, : \, v \hbox{ has an $n$-ancestor and }\parent^n(v)=u \}.
$$
If $v \in \child^n(u)$ for some $n\in \N$, we say that $v$ is a descendant of $u$.
\end{defi}

We will use frequently, and without mentioning it, the equivalence
between $u=\parent^n(v)$ and $v \in \child^n(u)$. Also, observe that
if the tree $T$ is unrooted, then $V^n=V$ for all $n \in \N$.

\subsection*{${\mathbf L}^p$ space of a tree.} Lastly, we define the
Banach spaces we will be dealing with. First, we mention that given a
countable set $V$, we will refer (abusing the notation) to the set
$\lambda=\{ \lambda_v \in \R_+ \, : \, v \in V\}$, indexed by $V$ as a
{\em sequence}, and we will denote it by  $\lambda=\{
\lambda_v\}_{v\in V}$. The Banach spaces we use throughout this paper
are always vector spaces over the complex numbers.

\begin{defi}
Let $1 \leq p < \infty$ and let $T=(V,E)$ be a directed tree. Let
$\lambda=\{ \lambda_v \}_{v \in V}$ be a sequence of positive
numbers. We denote by $\Lp$ the space of complex-valued functions $f:
V \to \C$ such that
$$
\sum_{v \in V} |f(v)|^p \lambda_v < \infty.
$$
\end{defi}
This is a Banach space if we endow it with the norm 
$$
\| f \|_p = \Bigg( \sum_{v \in V} |f(v)|^p \lambda_v \Bigg)^{1/p}.
$$
If $p=2$, this is also a Hilbert space (with the obvious inner
product). We do not consider in this paper the space ${\mathbf
  L}^\infty(T,\lambda)$, since it is not separable, regardless of the
choice of positive weights $\lambda$ (e.g., \cite[p. 115]{megginson}).
Observe that the only case of interest for us is when the sequence $\lambda$ is
composed of strictly positive numbers: if $\lambda_v=0$ for some $v$, it is
easily seen that the space can then be written as a direct sum of
(perhaps infinitely many) spaces on smaller directed trees.

The graph structure, obviously, has nothing to do with Banach space structure
of $\Lp$ itself. What is interesting in this setting, are the
operators that we can build here, which we define in the next section.

\section{The Shift Operator and its Hypercyclicity}\label{sec_shift}

We can now introduce one of the main objects of study of this paper.

\begin{defi}
Let $T=(V,E)$ be a directed tree, let $\lambda=\{ \lambda_v \}_{v \in
  V}$ be a positive sequence and let $1 \leq p < \infty$. The {\em
  shift} $S: \Lp \to \Lp$ is the operator defined as
\begin{equation*}
(S f) (v)=
\begin{cases}
f(\parent(v)), & \text{ if $v \neq \root$, and} \\
0, & \text{ if $v=\root$.}
\end{cases}
\end{equation*}
\end{defi}

The following proposition was established for weighted shifts (instead
of shifts on weighted spaces) in the case $p=2$ in \cite{JJS}. The
proof is the same for our case, but we include it here for the sake of
completeness.

\begin{proposition}
Let $T=(V,E)$ be a directed tree, let $\lambda=\{ \lambda_v \}_{v \in
  V}$ be a positive sequence and  let $1 \leq p < \infty$. The
operator $S: \Lp \to \Lp$ is bounded if and only if 
$$
\sup_{u \in V}  \sum_{v \in \child(u)} \frac{\lambda_v}{\lambda_u} < \infty.
$$
In either case, 
$$
\| S\| := \Bigg( \sup_{u \in V} \sum_{v \in \child(u)}
  \frac{\lambda_v}{\lambda_u} \Bigg)^{1/p}.
$$
\end{proposition}
\begin{proof}
Assume $\ds M:=\sup_{u \in V} \sum_{v \in \child(u)}
\frac{\lambda_v}{\lambda_u}   < \infty$. Let $f \in \Lp$. Then
\begin{eqnarray*}
\| S f \|_p^p 
&=& \sum_{v \in V} |(Sf)(v)|^p \lambda_v \\
&=& \sum_{\substack{v \in V, \\ v \neq \root}} |f(\parent(v))|^p \lambda_v \\
&=& \sum_{u \in V} |f(u)|^p \lambda_u \sum_{v \in \child(u)} \frac{\lambda_v}{\lambda_u} \\
& \leq &  M  \sum_{u \in V} |f(u)|^p \lambda_u \\
& = & M \| f \|_p^p.
\end{eqnarray*}
Hence, $\| S f \|_p \leq M^{1/p} \| f \|_p$. Thus $S$ is bounded and $\| S \| \leq M^{1/p}$.

On the other hand, assume $S$ is bounded. Let $u \in V$, and denote by
$\chi_{u}$ the characteristic function of the vertex $u$. Define
$f_u:=\frac{1}{\lambda_u^{1/p}} \chi_u$. Clearly,  $\| f_u \|_p=1$. Then,  if $v
\in V$, $v \neq \root$, we have $S f_u  (v) = f_u (\parent(v))$ and this
is zero unless $u=\parent(v)$. Therefore,
$$
\| S f_u \|_p^p 
= \sum_{v \in V} |(S f_u)(v)|^p \lambda_v 
= \sum_{\substack{v \in V, \\ v \neq \root}} |f_u(\parent(v))|^p \lambda_v 
= \sum_{v \in \child(u)} \frac{\lambda_v}{\lambda_u}.
$$
Hence 
$$
\sup_{u \in V} \sum_{ v \in \child(u)} \frac{\lambda_v}{\lambda_u}  \leq \| S\|^p,
$$
which implies that $\ds \sup_{u \in V} \sum_{v \in \child(u)}
\frac{\lambda_v}{\lambda_u}  < \infty$ and 
$$
\| S \| = \Bigg(\sup_{u \in V} \sum_{v \in \child(u)} \frac{\lambda_v}{\lambda_u} \Bigg)^{1/p},
$$
concluding the proof.
\end{proof}

Our first observation is that the shift $S$ cannot be hypercyclic if the tree $T$ has a
root (for example, if $S$ is the unilateral forward shift on the tree $\N_0$). Note that
\begin{equation*}
(S^n f) (v)=
\begin{cases}
f(\parent^n(v)), & \text{ if $v \in V^n$, and} \\
0, & \text{ if $v \notin V^n$,}
\end{cases}
\end{equation*}
for any $f \in \Lp$ and any $n \in \N$.

\begin{proposition}
Let $T=(V,E)$ be a directed tree, let $\lambda=\{\lambda_v \}_{v\in
  V}$ be a positive sequence and let $1\leq p<\infty$. If $T$ has a
root, then $S: \Lp \to \Lp$ is not hypercyclic.
\end{proposition}
\begin{proof}
Since $(Sf)(\root)=0$ for every $f \in \Lp$, it follows that $(S^n
f)(\root)=0$ for all $n \in \N$. Let $\chi_\root$ denote the
characteristic function of $\root$. If $f$ were a hypercyclic vector,
it would follow that there exists an increasing sequence $\{n_k\}$ of
positive integers such that
$$
\| S^{n_k} f - \chi_{\root} \|_p  \to 0,
$$
as $k \to \infty$. But since
\begin{eqnarray*}
(\lambda_{\root})^{1/p} 
&=& | 0 -1 | \ (\lambda_{\root})^{1/p} \\
&=& | (S^{n_k} f)(\root) - \chi_{\root}(\root) | \ (\lambda_{\root})^{1/p} \\
&\leq& \| S^{n_k} f - \chi_{\root} \|_p,
\end{eqnarray*}
this is a contradiction. Hence $f$ cannot be a hypercyclic vector and
$S$ cannot be a hypercyclic operator.
\end{proof}

The next observation is that the shift $S$ cannot be hypercyclic if the tree $T$ has a
vertex with outdegree larger than $1$.

\begin{proposition}
Let $T=(V,E)$ be a directed tree, let $\lambda=\{\lambda_v \}_{v\in
  V}$ be a positive sequence and let $1\leq p<\infty$. If $T$ has at
least one vertex of outdegree at least $2$, then $S: \Lp \to \Lp$ is
not hypercyclic.
\end{proposition}
\begin{proof}
Let $w$ be the vertex with outdegree $n$, with $n \geq 2$ and let
$v_1$ and $v_2$ be two different elements in $\child(w)$. Observe that
$\parent(v_1)=\parent(v_2)$ and hence
$$
(S^k f)(v_1)= f(\parent^k(v_1))=f(\parent^k(v_2))=(S^k f)(v_2)
$$ 
for all $k \in \N$ such that $v_1$ and $v_2$ have $k$-ancestors. If
$v_1$ and $v_2$ do not have $k$-ancestors, then $(S^k f)(v_1)= 0=(S^k
f)(v_2)$. Thus
$$
(S^k f)(v_1)= (S^k f)(v_2)
$$ 
for every $k \in \N$.

Let $\epsilon> 0$ such that $\epsilon <
((\lambda_{v_1})^{-1/p}+(\lambda_{v_2})^{-1/p})^{-1}$. If $f$ were a
hypercyclic vector for $S$, there would exist $N \in \N$ such that
$$
\| S^{N} f - \chi_{v_1} \|_p < \epsilon.
$$

We have then that
$$
| (S^{N} f)(v_1) - \chi_{v_1} (v_1) |  (\lambda_{v_1})^{1/p} \leq
\| S^{N} f - \chi_{v_1} \|_p < \epsilon,
$$
and
$$
| (S^{N} f)(v_2) - \chi_{v_1} (v_2) |  (\lambda_{v_2})^{1/p} \leq
\| S^{N} f - \chi_{v_1} \|_p < \epsilon.
$$
And hence, we have
$$
|  (S^{N} f)(v_1) - 1 | (\lambda_{v_1})^{1/p}   < \epsilon   \qquad
\hbox{ and } \qquad |  (S^{N} f)(v_2)   |(\lambda_{v_2})^{1/p}   <
\epsilon.
$$
Define  $z:=(S^{N} f)(v_1)=(S^{N} f)(v_2)$. We then have
$$
| z- 1 |  < \epsilon  (\lambda_{v_1})^{-1/p}  \qquad \hbox{ and }
\qquad | z  | < \epsilon  (\lambda_{v_2})^{-1/p}.
$$
But the first inequality above implies that
$$
1 - |z| < \epsilon  (\lambda_{v_1})^{-1/p} 
$$
and hence that
$$
1-\epsilon  (\lambda_{v_1})^{-1/p} < |z| < \epsilon  (\lambda_{v_2})^{-1/p} 
$$
which in turn implies that
$$
1-\epsilon  (\lambda_{v_1})^{-1/p} < \epsilon  (\lambda_{v_2})^{-1/p}, 
$$
which contradicts the choice of $\epsilon$. Hence $f$ cannot be a
hypercyclic vector and $S$ cannot be a hypercyclic operator.
\end{proof}

The previous two propositions imply that in order for $S$ to be
hypercyclic the tree $T$ cannot have a root and cannot have vertices
of outdegree larger than $1$. It is easy to see, then,  that the tree
must be either isomorphic to the directed graph $\ds
\left(\Z,\{(n,n+1) \, : \, n \in \Z\}\right)$ if $T$ has no leaf (see
the picture below),

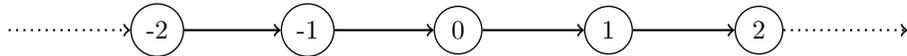
\begin{center}
\begin{tikzpicture}
\GraphInit[vstyle=Normal]
 \Vertex[x=0,y=0,empty=true,NoLabel]{A}
 \Vertex[x=2, y=0]{-2}
 \Vertex[x=4, y=0]{-1} 
 \Vertex[x=6, y=0]{0} 
 \Vertex[x=8, y=0]{1} 
 \Vertex[x=10, y=0]{2} 
  \Vertex[x=12,y=0,empty=true,NoLabel]{B}
 \SetUpEdge[style={->}]
 \Edges(-2,-1,0,1,2)
 \SetUpEdge[style={->,dotted}]
 \Edges(A,-2)
 \Edges(2,B)
\end{tikzpicture}
\end{center}
\medskip
\noindent or isomorphic to the directed graph $\ds \left(\N_0,\{
  (n+1,n) \, : \, n\in \N_0 \}\right)$ if $T$ has a leaf (see the picture below).
\medskip
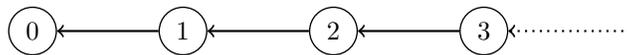
\begin{center}
\begin{tikzpicture}
\GraphInit[vstyle=Normal]
\Vertex[x=0, y=0]{0} 
\Vertex[x=2, y=0]{1} 
\Vertex[x=4, y=0]{2} 
\Vertex[x=6, y=0]{3} 
\Vertex[x=8, y=0,empty=true,NoLabel]{B}
 \SetUpEdge[style={<-}]
 \Edges(0,1,2,3)
 \SetUpEdge[style={<-,dotted}]
 \Edges(3,B)
\end{tikzpicture}
\end{center}
\medskip

In the former case, $S$ is just a bilateral shift on a weighted
$\ell^p(\Z)$ space and in the latter case $S$ is just a unilateral
backward shift on a weighted $\ell^p(\N_0)$ space. The hypercyclicity
in these two cases has been characterized by Salas~\cite{salas} (see
\cite[Theorems 1.38 and 1.40]{BaMa} and \cite[Theorems 4.3 and
4.12]{GEP} for alternative statements of Salas' result).

{\bf Note added:} The referee has kindly pointed out to us that in
\cite{GE}, Grosse-Erdmann has considered what he calls ``weighted
  pseudo-shifts'' on sequence spaces, and characterized their
  hypercyclicity. It is not hard to see that the shifts
  considered in Section~\ref{sec_shift} are weighted
  pseudo-shifts and, therefore, Grosse-Erdmann's characterization
  applies. Nevertheless, we should point out that, for the case where
  the structure of directed trees is available, our results give an
  easier-to-check formulation of the characterization, since we are
  able to tell that the shift may be hypercyclic only if the tree reduces
  to one of the cases already considered by Salas in \cite{salas}. We
  thank the referee for the observation.

\section{The Adjoint of the Shift Operator and the Backward
  Shift}\label{sec_adjoint}

Our goal in this section is to identify the adjoint operator of the
shift operator on a directed tree. The case for the Hilbert space
adjoint ($p=2)$ was done in \cite{JJS}. After we identify the adjoint,
we will define the backward shift.

The following result is standard in the theory of ${\mathbf L}^p$ spaces.

\begin{proposition}\label{prop:dual}
Let $T=(V,E)$ be a directed tree, let $\lambda=\{\lambda_v \}_{v\in V}$
be a positive sequence, let $1<p<\infty$ and let
$q=\frac{p}{p-1}$. For $g \in \Lq$ define $\Phi_g : \Lp \to \C$ as 
$$
\Phi_g(f) = \sum_{v\in V} f(v) g(v) \lambda_v.
$$
Then $\Phi_g$ is a bounded linear functional on $\Lp$. Conversely, if
$\Phi$ is a bounded linear functional on $\Lp$, there exists $g \in
\Lq$ such that $\Phi=\Phi_g$. Moreover, $\| \Phi_g \| = \| g \|_q$.
\end{proposition}

Henceforth, we identify the dual space of $\Lp$ with $\Lq$ and we will
use the identification of the vector $g \in \Lq$ with the functional
$\Phi_g$ on $\Lp$. We can now compute the adjoint of a shift.

\begin{proposition}\label{prop:adjoint}
Let $T=(V,E)$ be a directed tree, let $\lambda=\{ \lambda_v \}_{v \in
  V}$ be a positive sequence, let $1 < p < \infty$, and let
$q:=\frac{p}{p-1}$. Assume the operator $S: \Lp \to \Lp$ is
bounded. Then $S^* : \Lq \to \Lq$ is given by
$$
(S^* g)(u)=\sum_{v\in \child(u)} g(v) \frac{\lambda_v}{\lambda_u},
$$
for each $g \in \Lq$ and $u \in V$.
\end{proposition}
\begin{proof}
Let $g \in L^q$ and define $h: T \to \C$ as 
$$
h(u):= \sum_{v\in \child(u)} g(v) \frac{\lambda_v}{\lambda_u},
$$
where, as usual, we define a sum over an empty set to be zero. We
first show that $h \in \Lq$.

Let us define
$$
M:=\sup_{u \in V} \sum_{v\in \child(u)} \lambda_v \lambda_u^{-1}.
$$
Observe that, since $S$ is bounded, $M<\infty$.  Let $u \in V$. Clearly
$$
| h(u) | \leq \sum_{v\in \child(u)} |g(v)| \frac{\lambda_v}{\lambda_u}
$$
(note that if $u$ is a leaf then $h(u)=0$). By the classical
inequality of Jensen we have
$$
\left( \frac{\ds \sum_{v\in \child(u)} |g(v)|  \lambda_v
    \lambda_u^{-1}}{\ds \sum_{v \in \child(u)} \lambda_v
    \lambda_u^{-1}} \right)^q
\leq 
\frac{\ds \sum_{v\in \child(u)} |g(v)|^q \lambda_v \lambda_u^{-1}}{\ds
  \sum_{v \in \child(u)} \lambda_v \lambda_u^{-1}},
$$
which simplifies to
$$
\Bigg( \sum_{v\in \child(u)} |g(v)|  \lambda_v \lambda_u^{-1} \Bigg)^q
\leq
\Bigg(\sum_{v\in \child(u)} \lambda_v \lambda_u^{-1} \Bigg)^{q-1}
\Bigg( \sum_{v\in \child(u)} |g(v)|^q \lambda_v \lambda_u^{-1} \Bigg),
$$
and hence we have 
$$
|h(u)|^q \leq M^{q-1} \Bigg( \sum_{v\in \child(u)} |g(v)|^q \lambda_v
  \lambda_u^{-1} \Bigg).
$$
Therefore, multiplying by $\lambda_u$ and summing over all vertices $u \in V$, we get
$$
\sum_{u\in V}  |h(u)|^q \lambda_u 
\leq  M^{q-1} \sum_{u\in V} \Bigg( \sum_{v\in\child(u)} |g(v)|^q \lambda_v  \Bigg).
$$
The right-hand side of the previous expression is no larger than
$$
 M^{q-1} \sum_{v \in V} |g(v)|^q \lambda_v,
$$
and hence 
$$
\sum_{u\in V}  |h(u)|^q \lambda_u 
\leq  M^{q-1} \sum_{v \in V} |g(v)|^q \lambda_v
$$
which shows that $h \in\Lq$ and, in fact,  $\| h \|_q \leq M^{\frac{q-1}{q}} \| g \|_q$.

Now, let $f \in \Lp$. Then 
\begin{eqnarray*}
(S^* \Phi_g )(f)
&=&\Phi_g(Sf)\\
&=&\sum_{v \in V} (Sf)(v) g(v) \lambda_v  \\
&=&\sum_{\substack{v \in V, \\ v \neq \root}} f(\parent(v)) g(v)
\lambda_v \\
&=&\sum_{u \in V} f(u) \Bigg( \sum_{v \in \child(u)} g(v) \lambda_v \Bigg) \\
&=&\sum_{u \in V} f(u) \Bigg( \sum_{v \in \child(u)} g(v)
  \frac{\lambda_v}{\lambda_u} \Bigg) \lambda_u \\
&=&\sum_{u \in V} f(u) h(u)  \lambda_u \\
&=& \Phi_h (f),
\end{eqnarray*}
thus $S^* \Phi_g = \Phi_h$. If we identify, as usual, $\Phi_g$ with
$g$ and $\Phi_h$ with $h$ we have
$$
S^* g = h,
$$
which is what we wanted to prove.
\end{proof}

We now study the hypercyclicity of $S^*$. The first part of the
following proposition, for the case $p=2$, is proven implicitly in
\cite[Proposition 3.1.7]{JJS}.

\begin{proposition}\label{prop:leaf1}
Let $T=(V,E)$ be a directed tree, let $\lambda=\{ \lambda_v \}_{v \in
  V}$ be a positive sequence, let $1 < p < \infty$, and let
$q:=\frac{p}{p-1}$. Assume the operator $S: \Lp \to \Lp$ is
bounded. If the directed tree $T$ has a leaf, then the operator $S^*:
\Lq \to \Lq$ does not have dense range. Hence, $S^*$ is not
hypercyclic.
\end{proposition}
\begin{proof}
Let $w$ be a leaf and let $f:=\chi_{w}$. Since $f\neq 0$, we can
choose a functional $\Phi$ on $\Lp$ such that $\Phi(f)=1$. By
Proposition~\ref{prop:dual}, there exists $h \in \Lq$ such that
$\Phi=\Phi_h$. If $S^*$ had dense range, we could choose functions
$g_n \in \Lq$ such that $S^*g_n \to h$, and, hence, such that
$\Phi_{S^* g_n} \to \Phi_h$. But
$$
\Phi_{S^*g_n} (f)= \sum_{u \in V} f(u) (S^*g_n)(u)\lambda_u  =
(S^*g_n)(w) \lambda_w = \sum_{v \in \child(w)} g_n(v) \lambda_v=0,
$$
since $w$ is a leaf, and hence it has no children. But this implies that
$$
1=|0-1|= | \Phi_{S^* g_n} (f)- \Phi (f) | \leq \| \Phi_{S^* g_n} -
\Phi \| \ \| f \|_p \to 0,
$$
which is a contradiction. Hence $S^*$ does not have dense
range. Since every hypercyclic operator must have dense range, the
second part of the proposition follows.
\end{proof}

Observe that if we were to define an operator $S^*: \Lone \to \Lone$
by the expression
$$
(S^* f)(u)= \sum_{v \in \child u} f(v) \frac{\lambda_v}{\lambda_u},
$$
for each $u \in V$, then this operator would be bounded. Indeed, the
proof of Proposition~\ref{prop:adjoint} shows that if we set
$$
h(u):= \sum_{v\in \child(u)} g(v) \frac{\lambda_v}{\lambda_u},
$$
for each $u \in V$, then  $\| h \|_1 \leq \| g \|_1$. Hence $\| S^* g \|_1 \leq \| g\|_1$
and thus $S^*$ is a contraction. Therefore $S^*$ is not
hypercyclic. Nevertheless, we will denote this operator on $\Lone$ by
$S^*$, when the occasion arises.

The form of the operator $S^*$ on $\Lq$ suggests that a natural candidate for
study is the operator defined below. It will turn out that $S^*$ will
be unitarily equivalent to the following operator, with appropriate
weights. We will show this in the last section of this paper.

\begin{defi}
Let $T=(V,E)$ be a directed tree, let $\lambda=\{ \lambda_v \}_{v \in
  V}$ be a positive sequence and let $1 \leq q < \infty$. The {\em
  backward shift} is defined as the operator $B: \Lq \to \Lq$ given by
the expression
$$
(B f)(u)= \sum_{v \in \child(u)} f(v),
$$
for each $u \in V$. In the expression above, as it is usual, the sum
over an empty set is defined to be zero. 
\end{defi}

From now on, we will deal with the operator $B$, since the
hypercyclicity results we obtain are cleaner for $B$ than they are for
$S^*$. Let us show that, under certain conditions, $B$ is a bounded
operator. We denote by $\gamma(u)$ the cardinality of $\child(u)$.

\begin{proposition}
Let $T=(V,E)$ be a directed tree, let $\lambda=\{ \lambda_v \}_{v \in
  V}$ be a positive sequence and let $1 \leq  q < \infty$. If 
$$
\sup_{\substack{w \in V \\ w \neq \root}} \gamma(\parent(w))^{q-1}
\frac{\lambda_{\parent(w)}}{\lambda_w} < \infty,
$$
then the backward shift operator $B:\Lq \to \Lq$ is bounded.
\end{proposition}
\begin{proof}
Let 
$$
M:=\sup_{\substack{w \in V \\ w \neq \root}} \gamma(\parent(w))^{q-1}
\frac{\lambda_{\parent(w)}}{\lambda_w},
$$
and let $f \in \Lq$. Proceeding as in the proof of
Proposition~\ref{prop:adjoint}, by Jensen's inequality, for every $u \in V$ we have
$$
\Bigg( \sum_{v\in \child(u)} |f(v)| \Bigg)^q \leq (\gamma(u))^{q-1}
\sum_{v \in \child(u)} |f(v)|^q.
$$
It then follows that
\begin{eqnarray*}
\| B f \|^q_q 
&=& \sum_{u \in V} |Bf(u)|^q \lambda_u \\ 
&\leq & \sum_{u \in V} \Bigg( \sum_{v\in \child(u)} |f(v)| \Bigg)^q \lambda_u \\
&\leq &  \sum_{u \in V} (\gamma(u))^{q-1} \lambda_u \sum_{v \in \child(u)} |f(v)|^q \\
&=&  \sum_{\substack{w \in V \\ w \neq \root}} |f(w)|^q (\gamma(\parent(w)))^{q-1} \lambda_{\parent(w)}\\
&\leq & M   \sum_{\substack{w \in V \\ w \neq \root}} |f(w)|^q \lambda_{w} \\
& \leq & M \| f \|_q^q,
\end{eqnarray*}
and therefore $B$ is bounded.
\end{proof}

The special case of the unweighted space is simpler.

\begin{corollary}\label{cor_boundedB}
Let $T=(V,E)$ be a directed tree, and let $1 \leq q < \infty$. Let
$\lambda$ be the constant sequence defined by $\lambda_v=1$ for each
$v\in V$. If $q=1$, then the backward shift operator $B$ is bounded on
$\Lone$. If $1< q < \infty$, then $B$ is bounded on $\Lq$ if the set
$\{ \gamma(u) \, : \, u \in V\}$ is bounded; i.e., if the outdegrees
of the tree are bounded.
\end{corollary}

Let $\lambda$ be the constant sequence defined by $\lambda_v=1$ for
each $v\in V$. It turns out that the if $q=1$, the operator $B$ has
norm equal to one. If $1<q<\infty$, the condition in
Corollary~\ref{cor_boundedB} is not only sufficient, but also
necessary. We investigate these matters in a different paper \cite{CoMa}.

As it was the case in Proposition~\ref{prop:leaf1}, if the tree has
leaves, the backward shift operator is never hypercyclic.

\begin{proposition}\label{prop:leaf2}
Let $T=(V,E)$ be a directed tree, let $\lambda=\{ \lambda_v \}_{v \in
  V}$ be a positive sequence, let $1 \leq q < \infty$. Assume the
backward shift $B$ is bounded. If the directed tree $T$ has a leaf,
then the operator $B: \Lq \to \Lq$ is not hypercyclic.
\end{proposition}
\begin{proof}
Let $w \in V$ be a leaf.  Since the sum over an empty set is zero, for
every $g \in \Lq$ we have
$$
 (B g)(w) = \sum_{v\in \child(w)} g(v) =0.
$$ 
Hence $(B^n g)(w)=0$ for every $n \in \N$. If $g$ were
a hypercyclic vector for $B$, there would exist an increasing sequence
of natural numbers $\{ n_k \}$ such that $\| B^{n_k} g - \chi_{w}
\|_q \to 0$. But then
$$
\lambda_w^{1/q} = | 0 - 1 | \lambda_w^{1/q}  = | (B^{n_k} g )(w) -
\chi_{w} (w) | \lambda_w^{1/q} \leq \| B^{n_k} g - \chi_{w} \|_q \to 0 
$$
which is impossible.
\end{proof}

\section{Hypercyclicity of the Backward Shift: Rooted Directed
  Trees}\label{sec_root}

We need to distinguish two cases to study the hypercyclicity of
$B$. In this section we deal with the case where the tree $T$ has a
root. Observe that in this case, for every vertex $u \in V$, there
exists a unique path (in the underlying graph of the tree $T$)
starting from $\root$ and ending at $u$. We denote by $|u|$ the length
of such a path.

For every $u \in V$, and every $n\in \N$, we denote by $\gamma(u,n)$
the cardinality of the set $\child^n(u)$. Observe that $\gamma(u,n)>0$
for every $u \in V$ and every $n \in \N$ if $T$ is leafless.

For every $u \in V$, for $1 \leq q < \infty$, and every $n \in \N$ we
denote by $\Omega(u,n)$ the number
$$
\Omega(u,n):=
\frac{1}{(\gamma(u,n))^q}  \sum_{v \in \child^n(u)} \lambda_v. 
$$

The next theorem gives a sufficient condition for hypercyclicity of
$B$, in terms of the numbers defined above.

\begin{theorem}\label{th_hyp_root}
Let $T=(V,E)$ be a leafless directed tree with a root, let $\lambda=\{
\lambda_v \}_{v \in V}$ be a positive sequence, and let $1 \leq q <
\infty$. Assume the backward shift operator $B: \Lq \to \Lq$ is
bounded. If there exists an increasing sequence of natural numbers
 $\{ n_k \}$ such that, for all $u \in V$ we have
$$
\Omega(u,n_k) \to 0
$$ 
as $k \to \infty$, then $B$ is hypercyclic.
\end{theorem}
\begin{proof}
We will verify each of the conditions of the Hypercyclicity Criterion
(Theorem \ref{the_hc}). Define $X$ as the set $X:=\{ g \in \Lq \, : \,
g \hbox{ is finitely supported } \}$. Clearly $X$ is dense in $\Lq$.

First, it can easily be seen that, for every $f \in \Lq$ and every $u
\in V$, we have
$$
(B^n f)(u) = \sum_{v \in \child^n(u)} f(v).
$$

\begin{enumerate}
\item If $g$ is finitely supported, there exists $N \in \N$ such that $g(v)=0$ for
all $v$ with $|v|\geq N$. For all $u\in V$, we have that if $v
\in \child(u)$ then $|v|=|u|+1$ and hence that if $v \in \child^n(u)$, then
$|v|=|u|+n$. Therefore, for any $u \in V$, if $n \geq N$ and $v \in \child^n(u)$ then
$g(v)=0$. It follows that $(B^n g)(u)=0$ for all $u \in V$ as soon as
$n \geq N$. Thus the function $B^n g$ is identically zero if $n \geq
N$. Therefore $ B^{n_k} g \to 0$ for all $g \in X$, as $k \to \infty$.

\item Given $g \in X$ and $n \in \N$, we define the complex-valued function $T_n g$ as
$$
(T_n g)(v):=
\begin{cases}
\frac{1}{\gamma(\parent^n(v),n)} g (\parent^n(v)), & \text{ if $v \in V_n$, and} \\
0, & \text{ if $v \notin V_n$,}
\end{cases}
$$
where, as before, $V^n$ denotes the set of vertices that have $n$-ancestors. It follows that
\begin{eqnarray*}
\| T_{n} g \|_q^q
&=& \sum_{v \in V} |(T_n g)(v)|^q \lambda_v  \\
&=& \sum_{v \in V^n} |(T_n g)(v)|^q \lambda_v  \\
&=& \sum_{v \in V^n} \frac{1}{(\gamma(\parent^n(v),n))^q}  |g(\parent^n(v))|^q \lambda_v \\
&= & \sum_{u  \in V}  |g(u)|^q \frac{1}{(\gamma(u,n))^q}  \sum_{v \in \child^n(u)} \lambda_v  \\
&=& \sum_{u  \in V}  |g(u)|^q \, \Omega(u,n).
\end{eqnarray*}

Evaluating at the sequence $\{n_k\}$, remembering that $g$ is
finitely-supported (and hence the last expression has only
finitely-many summands) and recalling that for every $u \in V$ we
have $\Omega(u,n_k)\to 0$ as $j\to \infty$, it follows that
$$
T_{n_k} g  \to 0,
$$
as $k \to \infty$.

\item Lastly, we show that $B^n (T_n g) = g$ for all $g \in
  X$ and $n \in \N$. Indeed, if $g \in X$, and $u \in V$, then
\begin{eqnarray*}
(B^n(T_ng)) (u) 
&=&  \sum_{v \in \child^n(u)} (T_n g)(v) \\
&=&  \sum_{v \in \child^n(u)} \frac{1}{\gamma(\parent^n(v),n)} g (\parent^n(v)) \\
&=&  \sum_{v \in \child^n(u)} \frac{1}{\gamma(u,n)} g(u) \\
&=& g(u),
\end{eqnarray*}
as desired. Hence $B^{n_k} T_{n_k} g \to g$ as $k\to \infty$, for each $g \in X$.
\end{enumerate}
Therefore, by the Hypercyclicity Criterion, $B$ is hypercyclic.
\end{proof}

We will need the following definition to state some of the coming results.

\begin{defi}
Let $T$ be a leafless directed tree. We say that $T$ has a {\em free end}
if there exists a vertex such that all of its descendants have degree one.
\end{defi}

The fact that the following corollary does not apply for $q=1$ is not
surprising, given the comment after Corollary~\ref{cor_boundedB}.

\begin{corollary}\label{cor_suf1}
Let $T=(V,E)$ be a leafless directed tree with a root, let $1 < q <
\infty$, let $\lambda$ be the constant sequence defined by
$\lambda_v=1$ for each $v \in V$, and assume that the backward shift
$B$ is bounded on $\Lq$.  If the tree $T$ has no free end, then $B$ is
hypercyclic.
\end{corollary}
\begin{proof}
Let $u \in V$ be fixed.  Clearly, the sequence $\{ \gamma(u,n)\}_{n\in \N}$ is a
nondecreasing sequence of natural numbers. We claim that $\gamma(u,n)
\to \infty$ as $n \to \infty$. Indeed, the only possible way the
sequence would not go to infinity is if it became eventually
constant, say after $N$ steps. But this would mean that each of the
vertices in $\child^N(u)$ has the property that all of its descendants
have outdegree one, contradicting the fact that $T$ has no free ends.

It then follows that
$$
\Omega(u,n):= \frac{1}{(\gamma(u,n))^q}  \sum_{v \in \child^n(u)}
\lambda_v = \frac{1}{(\gamma(u,n))^{q-1}}  \to 0
$$
as $n \to \infty$. Applying the previous theorem to the full sequence
of natural numbers, we obtain that $B$ is hypercyclic.
\end{proof}

We now study a necessary condition for hypercyclicity of $B$.

\begin{theorem}\label{the_nec_root}
Let $T=(V,E)$ be a leafless directed tree with a root, let $\lambda=\{
\lambda_v \}_{v \in V}$ be a positive sequence, let $1 \leq q <
\infty$ and assume the backward shift $B$ is bounded on $\Lq$. If $B$
is hypercyclic, then for each $u \in V$ there exists an increasing
sequence of nonnegative integers $\{ n_k \}$ such that
$$
\sum_{v \in \child^{n_k}(u)} \lambda_v^{-1/q} \to \infty \qquad \hbox{ as } k \to \infty.
$$ 
\end{theorem}
\begin{proof}
Let $u \in V$ be fixed. We proceed inductively. Let $k \in \N$ and assume that
for each $j < k$ we have chosen $n_1< n_2 < \dots < n_j$ such that
$$
j< \sum_{v \in \child^{n_{j}}(u)} \lambda_v^{-1/q}
$$
(no assumption is needed if $k=1$).

Define $\ds \delta:= (k+\lambda_u^{-1/q})^{-1}$. Since $B$ is
hypercyclic, we can choose $f$ a hypercyclic vector such that
\begin{equation}\label{ineq1}
\| f \|_q < \delta.
\end{equation}

We  now choose an integer $n_k > n_{k-1}$ (or $n_1 \in \N$, if $k=1$) such that
\begin{equation}\label{ineq2}
\left\| B^{n_k} f - \chi_u \right\|_q < \delta,
\end{equation}
where $\chi_u$ is the characteristic function of $u$.

From inequality \eqref{ineq2}, we get
$$
\Big| \Big(B^{n_k}f - \chi_u \Big)(u) \Big|^q \lambda_u < \delta^q,
$$
which is
$$
\Bigg| \sum_{v \in \child^{n_k}(u)} f(v) -1 \Bigg|^q  \lambda_u < \delta^q,  
$$
and hence
$$
1 - \Bigg| \sum_{v \in \child^{n_k}(u)} f(v) \Bigg|  < \frac{\delta}{ \lambda_u^{1/q}}.
$$
Therefore
$$
1 -  \frac{\delta}{ \lambda_u^{1/q}} <  \sum_{v \in \child^{n_k}(u)} |f(v)|.
$$
By inequality \eqref{ineq1} we have that for each $v \in V$
$$
| f(v)|^q \lambda_v < \delta^q,
$$
and hence combining the last two inequalities we obtain
$$
1 -  \frac{\delta}{ \lambda_u^{1/q}} <  \sum_{v \in \child^{n_k}(u)}
|f(u)|  <  \sum_{v \in \child^{n_k}(u)}
\frac{\delta}{\lambda_v^{1/q}},
$$
which simplifies to
$$
k = \frac{1}{\delta} - \frac{1}{\lambda_u^{1/q}} <  \sum_{v \in
  \child^{n_k}(u)} \frac{1}{\lambda_v^{1/q}},
$$
which is what we wanted. Therefore, we have chosen an increasing
sequence of natural numbers $\{n_k\}$ such that
$$
k < \sum_{v \in \child^{n_k}(u)} \lambda_v^{-1/q}, 
$$
which proves the theorem.
\end{proof}

In the following corollary, we exclude the case $q=1$, given the
comment after Corollary~\ref{cor_boundedB}.

\begin{corollary}\label{cor_nec1}
Let $T=(V,E)$ be a leafless directed tree with a root, let $1 < q <
\infty$, let $\lambda$ be the constant sequence defined by
$\lambda_v=1$ for each $v\in V$, and assume that the backward shift
$B$ is bounded on $\Lq$. If the tree $T$ has a free end then $B$ is
not hypercyclic.
\end{corollary}
\begin{proof}
If $T$ has a free end, there exists a vertex such that all its
descendants have degree one. Let $w^*$ be one of these descendants. It
is then clear that $\gamma(w^*,n)=1$ for all $n$. But then, for any
sequence $\{n_k\}$ we have
$$
\sum_{v \in \child^{n_k}(w^*)} \lambda_v^{-1/q}=\gamma(w^*,n_k)=1.
$$
The previous theorem then assures that $B$ is not hypercyclic.
\end{proof}

Putting together Corollaries \ref{cor_suf1} and \ref{cor_nec1} we
obtain the following characterization of hypercyclicity of the backward
shift for the unweighted case.

\begin{corollary}
Let $T=(V,E)$ be a leafless directed tree with a root, let $1 < q <
\infty$, let $\lambda$ be the constant sequence defined by
$\lambda_v=1$ for each $v\in V$, and assume that the backward shift
$B$ is bounded on $\Lq$. The operator $B$ is hypercyclic if and only if
the tree $T$ has no free ends.
\end{corollary}

We have not been able to obtain a condition that is both necessary and
sufficient for hypercyclicity of $B$ on $\Lq$ for the case when the
sequence $\lambda$ is not constant. We leave the question open for
future research.

\section{Hypercyclicity of the Backward Shift: Unrooted Directed
  Trees}\label{sec_unroot}

We deal now with the case where the tree $T$ does not have a root. For each $u
\in V$ and $n \in \N$, recall that $\gamma(u,n)$ denotes the
cardinality of the set $\child^n(u)$ and that $\Omega(u,n)$ denotes the number
$$
\Omega(u,n):= \frac{1}{(\gamma(u,n))^q}  \sum_{v \in \child^n(u)} \lambda_v. 
$$
We also define for each $u\in V$ and each $n\in \N$ the number
$$
\Theta(u,n):= \left( \gamma(\parent^n(u),n) \right)^{q-1} \lambda_{\parent^n(u)}.
$$

We first give a necessary condition for hypercyclicity of the backward shift.

\begin{theorem}\label{th:unrooted}
Let $T=(V,E)$ be a leafless directed tree with no root, let
$\lambda=\{ \lambda_v \}_{v \in V}$ be a positive sequence, and let $1
\leq q < \infty$. Assume the backward shift operator $B: \Lq \to \Lq$
is bounded. If there exists an increasing sequence of natural numbers
$\{ n_k \}$ such that, for all $u \in V$ we have
$$
\Theta(u,n_k) \to 0 \quad \hbox{ and } \quad \Omega(u,n_k) \to 0
$$ 
as $k \to \infty$, then the operator $B: \Lq \to \Lq$ is hypercyclic.
\end{theorem}
\begin{proof}
As done in the proof of Theorem~\ref{th_hyp_root}, we apply the
Hypercyclicity Criterion (Theorem \ref{the_hc}) to $B$. Again, $X$
denotes the set $X:=\{ g \in \Lq \, : \, g \hbox{ is finitely
  supported } \}$, which is dense in $\Lq$. Also, recall that for
every $g \in \Lq$ and every $u \in V$ we have
$$
(B^n g)(u) = \sum_{v \in \child^n(u)} g(v).
$$

\begin{enumerate}
\item Let $g \in X$. For every $u \in V$ we have
$$
|(B^n g)(u)|^q  \leq \Bigg(  \sum_{v \in \child^n(u)} |g(v)| \Bigg)^q
$$
As before, by Jensen's inequality, we have
$$
\Bigg( \sum_{v\in \child^n(u)} |g(v)|  \Bigg)^q
\leq
\Big(\gamma(u,n) \Big)^{q-1} \Bigg( \sum_{v\in \child^n(u)} |g(v)|^q \Bigg),
$$
and hence
$$
|(B^n g)(u)|^q \leq \sum_{v\in \child^n(u)}  (\gamma(u,n))^{q-1}  |g(v)|^q.
$$
Multiplying by $\lambda_u$ and summing over all $u \in V$, we obtain
\begin{eqnarray*}
\| B^n g \|_q^q
&=& \sum_{u \in V} |(B^n g)(u)|^q \lambda_u \\
&\leq& \sum_{u \in V} \Bigg( \sum_{v\in \child^n(u)} (\gamma(u,n))^{q-1} \lambda_u  |g(v)|^q \Bigg) \\
&=& \sum_{w \in V} \left( \gamma(\parent^n(w),n) \right)^{q-1} \lambda_{\parent^n(w)}  |g(w)|^q\\
&=& \sum_{w \in V} \Theta(w,n)  |g(w)|^q.
\end{eqnarray*}
Since $g \in X$, this sum is finite. Evaluating at the sequence $n_k$
and taking the limit as $k \to \infty$, we obtain
$$
B^{n_k} g  \to 0.
$$
This shows that the first condition of the Hypercyclicity Criterion
(Theorem~\ref{the_hc}) is satisfied.
\item Now, given $g \in X$ and $n \in \N$, we define the
  complex-valued function $T_n g$ as
$$
(T_n g)(v):= \frac{1}{\gamma(\parent^n(v),n)} g (\parent^n(v)),
$$
for $v \in V$. The rest of the proof of this part is the same as that
of Part (2) of Theorem~\ref{th_hyp_root}.
\item This is the same as Part (3) of Theorem~\ref{th_hyp_root}.
\end{enumerate}
Since all the conditions of the Hypercyclicity Criterion are
satisfied, the operator $B$ is hypercyclic.
\end{proof}

We show an example of an unrooted tree where hypercyclicity occurs.

\begin{example}
Let $1\leq q < \infty$, let  $r \in \N$ and let $s \in \R$ with $s >r^{q-1}$. Let $T$ be the
infinite unrooted leafless directed tree such that each vertex has
outdegree $r$. Observe that, for each $v \in V$ and $n \in \N$, we have
$\gamma(v,n)=r^n$.

Select an arbitrary fixed vertex and call it $w^*$.  Consider the set
$H$ defined as the set if all vertices that share a common
$n$-ancestor with $w^*$, that is
$$
H:= \{  w \in V \, : \, w \in \child^n(\parent^n(w^*)) \hbox{ for some } n \in \N \}.
$$
For each $u \in V$, set $\ds \lambda_u=\frac{1}{s^d}$, where
$d=\dist(u,H)$. Then the operator $B: \Lq \to \Lq$ is hypercyclic.
\end{example}
\begin{proof}
Fix $u \in V$ with $\ds \lambda_u=\frac{1}{s^d}$ for some $d \in \N_0$. Also, let $n>d$.

We have  that 
$$
\lambda_{\parent^n(u)}=\frac{1}{s^{n \pm d}},
$$ 
where the plus sign corresponds to the case where $u$ has a descendant
in $H$ and the minus sign corresponds to the case where $u$ has an
ancestor in $H$. Then
$$
\Theta(w,n)=(\gamma(\parent^n(w),n))^{q-1} \lambda_{\parent^n(w)} =
r^{n(q-1)} \frac{1}{s^{n\pm d}} =\left(\frac{r^{q-1}}{s}\right)^n
\frac{1}{s^{\pm d}}
$$
which goes to zero as $n$ goes to infinity.

Also, for each $v \in \child^n(u)$ we have
$$
\lambda_v=\frac{1}{s^{n \pm d}}
$$ 
where the plus sign corresponds to the case where $u$ has an ancestor
in $H$ and the minus sign corresponds to the case where $u$ has a
descendant in $H$. Then
$$
\Omega(u,n)=\frac{1}{(\gamma(u,n))^q} \sum_{v \in \child^n(u)}
\lambda_v = \frac{1}{r^{nq}} \sum_{v \in \child^n(u)} \frac{1}{s^{n
    \pm d}} =\frac{1}{r^{nq}} r^n \frac{1}{s^{n \pm d}} =
\frac{1}{(r^{q-1} s)^n} \frac{1}{s^{\pm d}}
$$
which goes to zero as $n$ goes to infinity, since $r^{q-1} s >1$. Hence, by the previous
theorem applied to the full sequence $n_k=k$, the operator $B: \Lq \to
\Lq$ is hypercyclic.
\end{proof}

It is easy to see that the conditions in the hypothesis of
Theorem~\ref{th:unrooted} reduce to the conditions given by Salas~\cite{salas}, if
$T$ is the rootless tree $\ds \left(\Z,\{(n,n+1) \, : \, n \in
  \Z\}\right)$. In this case, Salas' work shows the conditions are
necessary and sufficient. 

It is not hard to obtain a necessary condition for hypercyclicity in
the general case for the case of the unrooted tree, in the style of
Theorem~\ref{the_nec_root}. Nevertheless, we could not find any other
family of trees $T$ or weights $\lambda$ where such a condition was also sufficient (except,
of course, the case of the trees already described by Salas). Is there a
necessary and sufficient condition for hypercyclicity in this case? Or
for an ``interesting'' family? We leave the question open for future research.

\section{Relation between $B$ and $S^*$}

We now show that the operators $B$ and $S^*$ are indeed unitarily
equivalent. We can then translate the sufficient condition for
hypercyclicity we found for $B$ to the operator $S^*$. First, we need
a lemma.

\begin{lemma}\label{lem_isom}
Let $T=(V,E)$ be a directed tree and let $1\leq q<\infty$. Let
$\lambda=\{\lambda_u\}_{u\in V}$ and $\mu=\{\mu_u\}_{u\in V}$ be
sequences of positive numbers such that $\mu_u \lambda_u^{q-1}=1$ for
every $u \in U$. Define the operator $\Phi: \Lqmu \to \Lq$ as 
$$
(\Phi f)(u):=\frac{f(u)}{\lambda_u}.
$$ 
Then $\Phi$ is an isometric
isomorphism between $\Lqmu$ and $\Lq$.
\end{lemma}
\begin{proof}
Clearly $\Phi$ is linear. To see that $\Phi$ is isometric, let $f \in \Lqmu$. Then
$$
\| f \|_q^q = \sum_{u \in V} | f(u)|^q \mu_u,
$$
and
$$
\| \Phi f \|_q^q = \sum_{u \in V} | (\Phi f)(u)|^q \lambda_u = \sum_{u
  \in V} \frac{|f(u)|^q}{\lambda_u^q} \lambda_u =  \sum_{u \in V}
\frac{|f(u)|^q}{\lambda_u^{q-1}}.
$$
Since $\mu_u = \frac{1}{\lambda_u^{q-1}}$, it follows that $\| \Phi f
\|_q= \| f \|_q$. We now show $\Phi$ is surjective. Let $g \in
\Lq$. Define $f(u):=\lambda_u g(u)$. We show that $f \in
\Lqmu$. Indeed, the hypothesis implies that $\lambda_u^q \mu_u =
\lambda_u$ and hence we have
$$
\| f \|_q^q = \sum_{u \in V} | f(u)|^q \mu_u = \sum_{u \in V}
\lambda_u^q |g(u)|^q \mu_u =  \sum_{u \in V}  |g(u)|^q \lambda_u <
\infty,
$$
since $g \in \Lq$. Clearly, $\Phi f = g$, which concludes the proof.
\end{proof}

As we mentioned before, $S^*$ and $B$ turn out to be unitarily
equivalent. Recall that in Section~\ref{sec_adjoint} we defined $S^*$,
even in the case $q=1$.

\begin{theorem}\label{the_unit_equiv}
Let $1\leq  q<\infty$ and let $T=(V,E)$ be a directed tree. Let
$\lambda=\{\lambda_u\}_{u\in V}$ and $\mu=\{\mu_u\}_{u\in V}$ be
sequences of positive numbers such that $\mu_u \lambda_u^{q-1}=1$ for
every $u \in U$. Then $S^*: \Lq \to \Lq$ is unitarily equivalent to
$B: \Lqmu \to \Lqmu$.
\end{theorem}
\begin{proof}
Let $\Phi$ be as in Lemma~\ref{lem_isom}. We will show that $S^* \Phi
= \Phi B$.

Let $f \in \Lqmu$. For each $u \in V$ we have
$$
(\Phi B f)(u)= \frac{1}{\lambda_u} (Bf)(u) = \frac{1}{\lambda_u}
\sum_{v\in \child(u)} f(v),
$$
and, Proposition~\ref{prop:adjoint} gives
$$
(S^* \Phi f)(u) = \sum_{v \in \child(u)} (\Phi f)(v)
\frac{\lambda_v}{\lambda_u} = \sum_{v \in \child(u)}
\frac{f(v)}{\lambda_v} \frac{\lambda_v}{\lambda_u} = \sum_{v \in
  \child(u)} f(v)\frac{1}{\lambda_u}.
$$
Hence $S^* \Phi = \Phi B$, as desired. Since $\Phi$ is an isometric
isomorphism, the result follows.
\end{proof}

We can use the previous theorem to give a sufficient condition for the
hypercyclicity of $S^*$ in the case of the rooted tree.

\begin{theorem}
Let $T=(V,E)$ be a leafless directed tree with a root, let $\lambda=\{
\lambda_v \}_{v \in V}$ be a positive sequence, and let $1 < q <
\infty$. If there exists an increasing sequence of natural numbers $\{
n_k \}$ such that, for all $u \in V$ we have
$$
\frac{1}{(\gamma(u,n_k))^q}  \sum_{v \in \child^{n_k}(u)} \lambda_v^{1-q} \to 0
$$ 
as $k \to \infty$, then $S^*: \Lq \to \Lq$ is hypercyclic.
\end{theorem}
\begin{proof}
For each $u \in V$, define $\mu_u$ as $\mu_u=\lambda_u^{1-q}$. By
Theorem~\ref{th_hyp_root}, we have that $B: \Lqmu \to \Lqmu$ is
hypercyclic. But since $\mu_u \lambda_u^{q-1}=1$ for each $u \in V$,
Theorem~\ref{the_unit_equiv} implies that the operator $S^*$ is
unitarily equivalent to $B$. Hence $S^*$ is hypercyclic.
\end{proof}

Similar results can be obtained for $S^*$ in the case of the unrooted
tree and for necessary conditions for hypercyclicity in the case
of the rooted tree. We leave them as an exercise for the interested reader.

\end{document}